\documentclass{amsart}

\usepackage{amsthm,amsmath,amssymb}
\newtheorem{theorem}{Theorem}[section]
\newtheorem{lemma}[theorem]{Lemma}
\newtheorem{proposition}[theorem]{Proposition}

\theoremstyle{definition}

\newtheorem{remark}[theorem]{Remark}
\newtheorem{remarks}[theorem]{Remarks}

\numberwithin{equation}{section}

\DeclareMathOperator{\g}{\mathsf{g}}
\DeclareMathOperator{\ord}{ord}
\DeclareMathOperator{\supp}{supp}

\newcommand{\Z}{\mathcal{Z}}

\begin{document}
\title[On the Harborth constant of $C_3 \oplus C_{3n}$]{On the Harborth constant of $C_3 \oplus C_{3n}$}
\author{P. Guillot \and L. E. Marchan \and  O. Ordaz  \and W. A. Schmid \and H. Zerdoum}

\address{(P. G. \& W. A. S. \& H. Z.) Universit\'e Paris 13, Sorbonne Paris Cit\'e, LAGA, CNRS, UMR 7539, Universit\'e Paris 8, 
F-93430, Villetaneuse, France, and 
Laboratoire Analyse, G\'eom\'etrie et Applications (LAGA, UMR 7539), COMUE Universit\'e Paris Lumi\`eres, Universit\'e Paris 8, 
CNRS, 93526 Saint-Denis cedex, France}
\email{\texttt{philippe.guillot@univ-paris8.fr}} 
\email{\texttt{schmid@math.univ-paris13.fr}}
\email{\texttt{hanane\_zerdoum@yahoo.fr}}

\address{(L. E. M.) Escuela Superior Polit\'ecnica del Litoral, ESPOL, Facultad de ciencias naturales y matem\'atica. Campus Gustavo Galindo, km 30.5, v\'ia Perimetral, P.O. Box 09-01-5863, Guayaquil, Ecuador.}
\email{\texttt{lmarchan@espol.edu.ec}}

\address{(O.O.) Escuela de Matem\'aticas y Laboratorio MoST, Centro ISYS, Facultad de Ciencias,
Universidad Central de Venezuela, Ap. 47567, Caracas 1041--A, Venezuela}
\email{\texttt{oscarordaz55@gmail.com}}



\begin{abstract}
For a finite abelian group $(G,+, 0)$ the Harborth constant $\g(G)$ is the smallest integer $k$ such that each squarefree sequence over $G$ of length $k$, equivalently each subset of $G$ of cardinality at least $k$,  has a subsequence of length $\exp(G)$ whose sum is $0$. In this paper, it is established that  $\g(G)= 3n + 3$ for prime $n \neq 3$ and $\g(C_3 \oplus C_{9})= 13$.  
\end{abstract}

\maketitle

\section{Introduction}

For $(G,+,0)$ a finite abelian group, a zero-sum constant of $G$ is often defined as the smallest integer $k$ such that each set (or sequence, resp.) of elements of $G$ of cardinality (or length, resp.) at least $k$ has a subset (or subsequence, resp.) whose elements sum to $0$, the neutral element of the group, and that possibly fulfills some additional condition (typically on its size). We refer to the survey article \cite{gaogersurvey} for an overview of zero-sum constants of this and related forms. It is technically advantageous to work with squarefree sequences, that is, sequences where all terms are distinct, instead of sets. 

Harborth \cite{harborth} considered the constants that arise, for sequences and for squarefree sequences, when the additional condition on the subsequence is that its length is equal to the exponent of the group. His original motivation was a problem on lattice points. Considering these constants can be seen as an extension of the problem settled in the Theorem of Erd{\H o}s--Ginzburg--Ziv \cite{egz} from cyclic groups to general finite abelian groups.  

The constant nowadays called the Harborth constant of $G$, denoted $\g(G)$, is the constant that arises when considering squarefree sequences in the above mentioned problem. 
That is, $\g(G)$ is the smallest integer $k$ such that every squarefree sequences over $G$ of length at least $k$ has a subsequence of length $\exp(G)$ that sums to $0$. The exact value of $\g(G)$ is only known for a few types of groups. We refer to the monograph by Bajnok \cite{bajnok}, in particular Chapter F.3, for a detailed exposition. We recall some known results that are relevant for our current investigations.  

\subsection{Some known results}

For $G$ an elementary $2$-group, that is, the exponent of the group is $2$, the problem admits a direct solution:  there are no squarefree sequences of length $2$ that sum to $0$, and it follows that $\g(G)=|G|+1$ as there are no squarefree sequences of length strictly greater than the cardinality of $G$ and therefore  the condition is vacuously true for these sequences. 
For elementary $3$-groups the problem of determining $\g(G)$ is particularly popular as it is equivalent to several other well-investigated problems such as cap-sets and sets without $3$-term arithmetic progressions. Nevertheless, the exact value for elementary $3$-groups is only known up to rank $6$ (see \cite{edeletal} for a detailed overview and \cite{potechin} for the result for rank $6$); recently Ellenberg and Gijswijt \cite{ellenberg-g}, building on the work of Croot, Lev, and Pach \cite{crootetal}, obtained a major improvement  on asymptotic upper bounds for elementary $3$-groups. 

If $G$ is a cyclic group, then the problem again admits a direct solution: the only squarefree sequence of length exponent is the one containing each element of the group $G$ and it  suffices to check whether the sum of all elements of $G$ is $0$ or not. More concretely, for $n$ a strictly positive integer and $C_n$ a cyclic group of order $n$, one has:
\[
\g(C_{n})  = 
\begin{cases}  
 n & \text{if  $n$ is odd} \\
 n+1 & \text{if  $n$ is even} 
\end{cases}
\]

For groups of rank two the problem of determining $\g(G)$ is wide open. 
It is know that $\g(C_p \oplus C_p)=2p-1$ for prime $p \ge 47$ and for $p\in \{3,5,7\}$. The latter is due to Kemnitz \cite{kemnitz}, the former due to Gao and Thangadurai \cite{gaothanga}, with an additional minor improvement from the original $p \ge 67$ to $p\ge 47$ in \cite{ggs}. Furthermore, Gao and Thangadurai \cite{gaothanga} determined $\g(C_4 \oplus C_4)= 9$ and then made the following conjecture: 
\[
\g(C_{n} \oplus C_{n})=
\begin{cases}
 2n-1 & \text{if  $n$ is odd}\\
 2n+1 &  \text{if  $n$ is even}
 \end{cases}
\]

Moreover, Ramos and some of the present authors \cite{mors1} determined the value for groups of the form $C_2 \oplus C_{2n}$: 
\[
\g(C_{2} \oplus C_{2n})=
\begin{cases}
 2n+3 & \text{if  $n$ is odd}\\
 2n+2 & \text{if  $n$ is even}
\end{cases}
\]

Finally, Kiefer \cite{kiefer} (also see \cite[Proposition F.104]{bajnok}) showed that $\g(C_3 \oplus C_{3n}) \ge 3n +3$ for $n \ge 2$, 
which for $n$ odd, is larger by one then what might be expected (we refer to Section \ref{ssec_lower} for further details).  

\subsection{Main result}
In the current paper, we determine $\g(C_3 \oplus C_{3n})$ when $n$ is a prime number. It turns out that the bound by Kiefer is usually sharp, yet there is one exception, namely $n=3$. 
Specifically, we will show:
\[
\g(C_3 \oplus C_{3n}) = 
\begin{cases}
  3n + 3 & \text{if  $n\neq 3$ is prime} \\
  3n + 4 & \text{if $n=3$}
\end{cases}
\]
The proof  makes use of various addition theorems, namely the Theorems of Cauchy--Davenport, Dias da Silva--Hamidoune, and Vosper. These are applied to `projections' of the set to the subgroup $C_n$ of $C_3 \oplus C_{3n}$. This is a reason why our investigations are limited to groups where $n$ is prime. We also obtain some results by computational means. In particular, 
we confirm the conjecture by Gao and Thangadurai that we mentioned above for $C_6 \oplus C_6$.

\section{Preliminaries}
\label{sec_prel}

The notation used in this paper  follows \cite{geroldinger_barc}. We recall some key notions and results. For $a,b \in \mathbb{R}$ the interval of \emph{integers} is denoted by $[a,b]$, that is, $[a,b] = \{ z \in \mathbb{Z} \colon a \le z \le b \}$. A cyclic group of order $n$ is denoted by $C_n$. 

Let $G$ be a finite abelian group; we use additive notation. There are uniquely determined non-negative integers $r$ and  $1 < n_{1} \mid \dots \mid n_{r}$ such that $G \cong C_{n_{1}} \oplus \dots \oplus C_{n_{r}}$. The integer $r$ is called the rank of $G$. Moreover, if $|G|> 1$, then  $n_r$ is the exponent of $G$, denoted $\exp(G)$;  for a group of cardinality $1$ the exponent is $1$. 

A sequence over $G$ is an element of the free abelian monoid over $G$. Multiplicative notation is used for this monoid and its neutral element, the empty sequence, is denoted by $1$.
That is, for $S$ a sequence over $G$, for each $g \in G$  there exists a unique non-negative integer $v_{g}$ such that $S=\prod_{g \in G} g^{v_{g}}$; we call $v_{g}$ the multiplicity of $g$ in $S$. For each sequence $S$ over $G$ there exist not necessarily distinct  $g_{1}, \dots, g_{\ell}$ in  $G$ such that  $S=g_{1} \dots g_{\ell}$; these elements are determined uniquely up to ordering.  The sequence $S$ is called squarefree if $v_g \le 1$ for each $G$, equivalently, all the $g_{i}$ are distinct.

The length of $S$ is  $\ell=\sum_{g \in G}v_g$; it is denoted by $|S|$. The sum of $S$ is  $\sum_{i=1}^\ell g_{i}= \sum_{g \in G}v_gg$; it is denoted by $\sigma(S)$.  
The support of the  sequence  $S$, denoted $\supp(S)$, is the \emph{set} of elements appearing in $S$, that is, $\supp(S)=\{g \in G\colon v_{g} > 0\}$. 
A subsequence of $S$ is a sequence $T$ that divides $S$ in the monoid of sequences, that is $T=\prod_{i \in I} g_{i}$ for some $I \subset [1,\ell]$. Moreover, $T^{-1}S$ denotes the sequence fulfilling $(T^{-1}S)T=S$. 

Let $G$ and $G'$ be two groups and let $f$ be a map from $G$ to $G'$. We denote by $f$ also the homomorphic extension of $f$ to the monoid of sequences, that is, if $S=g_{1} \dots g_{\ell}$ is  a sequence over $G$, then $f(S)=f(g_{1}) \dots f(g_{\ell})$ is a sequence over  $G'$. Note that $|S| = |f(S)|$ always holds, even if the map $f$ is not injective. This highlights a difference between working with sequences and working with sets.  The image of a squarefree  sequence might not be squarefree anymore, but it always has the same length as the original sequence.
By contrast,  for $A \subset G$ a subset   $f(A) = \{f(a) \colon a \in A\}$ can have a  cardinality strictly smaller than $A$. 

If $f$ is a group homomorphism, then $\sigma(f(S))  =  f(\sigma(S))$. In particular, if $f$ is an isomorphism, then $S$ has a zero-sum subsequence of length $k$ if and only if 
$f(S)$ has a zero-sum subsequence of length $k$.  Moreover, for $g\in G$ and $S= g_1 \dots g_l$, the sequence $(g+g_1) \dots (g+ g_\ell)$ is denoted by $g+S$. Note that $S$ has a zero-sum subsequence of length $ \exp(G)$ if and only if $g+S$ has  a zero-sum subsequence of length $\exp(G)$.

The set $\Sigma (S)=\{\sigma(T)\colon 1\neq T \mid S\}$ is the set of (nonempty) subsums of $S$.  A sequence is called zero-sum free if $0 \notin \Sigma(S)$.  
Moreover, for a non-negative integer $h$, let $\Sigma_h(S)=\{ \sigma(T) \colon T \mid S \text{ with } |T| = h \}$ denote the set of $h$-term subsums. 
These notations are also used for sets with the analogous meaning. 

Using this notation the definition of the Harborth constant can be stated as follows: $\g(G)$ is the smallest integer $k$ such that for each squarefree sequence $S$ over $G$ with length $|S|\ge k$ one has $0 \in \Sigma_{\exp(G)}(S)$. 
We also need the Davenport constant $\mathsf{D}(G)$, which is defined as  the smallest integer $k$ such that for each  sequence $S$ over $G$ with length $|S|\ge k$ one has $0 \in \Sigma(S)$.

Let $A$ and $B$ be subsets of $G$. Then  $A+B$ denotes the set $\{a+b \colon a\in A, \, b \in B\}$, called the sumset of $A$ and $B$; moreover 
$A \widehat{+} B $ denotes the set $ \{ a+b \colon a\in A, \, b \in B, \, a \neq b \}$, called the restricted sumset of $A$ and $B$. 
Note that $A \widehat{+} A = \Sigma_2 (A)$.  

We recall some well-known results on set-addition in cyclic groups of prime order. 
We start with the classical Theorem of Cauchy--Davenport (see for example \cite[Theorem 6.2]{grynkiewicz_book}).

\begin{theorem}[Cauchy--Davenport]
\label{CDthm} 
Let $p$ be a prime number and let $A, B \subset C_p$ be non-empty sets, then:
\begin{equation*}
|A+B|  \ge \min \{p, |A|+|B|-1\} 
\end{equation*}
\end{theorem}
This yields immediately that  for non-empty sets $A_1, \dots, A_h \subset C_p$ one has:  
\begin{equation*}
|A_1 + \dots  +  A_h| \ge \min \bigl\{p ,  \sum_{i=1}^h |A_i| - (h-1)\bigr\}
\end{equation*}

The associated inverse problem, that is, the characterization of sets where the bound is sharp, is solved by the Theorem of Vosper (see for example \cite[Theorem 8.1]{grynkiewicz_book}). 
\begin{theorem}[Vosper]
\label{vosper}
Let $p$ be a prime number and let $A, B \subset C_p$. 
Suppose that $|A|, |B| \ge 2$ and $|A+B|= |A|+|B|-1$. 
\begin{itemize}
\item If $|A+B| \le p - 2$, then $A$ and $B$ are arithmetic progressions with common difference, that is there is some $d \in C_p$ and there are $a,b \in C_p$ such that 
	  $A = \{a+id \colon  i \in  [0,|A|-1]\}$ and $B = \{b+id \colon  i \in  [0,|B|-1]\}$.
\item If $|A+B|=p -1$, then  $A = \{c - a \colon a \in C_p \setminus B\}$ for some $c \in C_p$. 
\end{itemize}
\end{theorem}

We also need the analogue of the Theorem of Cauchy--Davenport for restricted set addition. It is called the Theorem of Dias da Silva--Hamidoune (see for example \cite[Theorem 22.5]{grynkiewicz_book}). 
\begin{theorem}[Dias da Silva--Hamidoune]
\label{thmDH}
Let $p$ be a prime number.  Let $A \subset C_p$ be a non-empty subset and let $h \in [1, |A|]$. Then:
\begin{equation*}
|\Sigma_h (A)| \ge \min \bigl\{ p, h (|A|-h) + 1 \bigr\}
\end{equation*}
\end{theorem}

We end this section with two technical lemmas. The first asserts that, except for some corner-cases, 
the difference of an arithmetic progression in a cyclic group of prime order is, up to sign,  uniquely determined. 
We include a proof as we could not find a suitable reference. 

\begin{lemma}
\label{lem_diff_ap}
Let $p \geq 5$ be a prime number and let $A \subset C_p$ be a set such that $|A|=k$ with $2 \le k \le p-2$. Assume that  
$A$ is an arithmetic progression, that is, there are some $r,a \in C_{p}$, such that $A=\{a+ir\colon i \in [0,k-1] \}$. The difference $r$ is determined uniquely up to sign, that is, 
if there are some $s,b \in C_{p}$ such that $A=\{b+is\colon i \in [0,k-1] \}$, then $s \in \{r,-r\}$. 
\end{lemma}
\begin{proof}
Since $A$ is an arithmetic progression with difference $r$ if and only if the the complement of $A$ in $C_p$ is an arithmetic progression with difference $r$, 
we can assume that $|A| \le \frac{p-1}{2}$.  Let $e$ be some non-zero element of $C_p$.

As the problem is invariant under affine transformations, we can assume without loss of generality that $A= \{0, e, 2e, \dots , (k-1)e\}$.
Suppose for a contradiction that $A= \{a+ir \colon i \in [0,k-1] \}$ with $a,r \in C_p$ and $r \notin \{e,-e\}$. Without loss of generality we can assume that $r= r' e$ with $r' \in [2, \frac{p-1}{2}]$.

As 
\[
k-1 < k-1+r' \le \frac{p-1}{2}+\frac{p-1}{2}-1=p-2 < p,
\] it follows that $(k-1)e + r \notin A$. It follows that $(k-1)e$ is also the last element of the arithmetic progression $A$ 
when represented with respect to the difference $r$. That is, $(k-1)e = a + (k-1)r$.  

The same reasoning shows that, when removing the element $(k-1)e$ from $A$ then $(k-2)e$ is the last element of arithmetic progression $A \setminus \{(k-1)e\}$ both with respect to the difference $r$ and $e$. Consequently,  $(k-2)e+ r = (k-1)e$. Thus, $r=e$.  
\end{proof}

When trying to establish the existence of zero-sum  subsequences  whose length is close to that of the full sequence, 
it can be  advantageous to work instead with the few elements of the sequence not contained in the putative subsequence. 
We formulate the exact link in the lemma below.   

\begin{lemma}
\label{complement}
Let $G$ be a finite abelian group. Let $0 \le r\le k$.
The following statements are equivalent. 
\begin{itemize}
\item Every squarefree sequence $S$ over $G$ of length $k$ has a subsequence $R$ of length $r$ with $\sigma(S) = \sigma (R)$. 
\item Every squarefree sequence $S$ over $G$ of length $k$ has a zero-sum subsequence $T$ of length $k-r$. 
\end{itemize}
\end{lemma}
\begin{proof}
Let $S$ be a squarefree sequence of length $k$.
Now, let $R$ be a subsequence of length $r$ with $\sigma(R)= \sigma(S)$. Then the sequence $T=R^{-1}S$ is a sequence of length $k-r$ with sum $\sigma(S)-\sigma(R) = 0$. 
Conversely,  let $T$ be a zero-sum subsequence of length $k-r$. Then the sequence $R=T^{-1}S$ is a sequence of length $k-(k-r)=r$ with sum $ \sigma(R)=\sigma(S)-\sigma(T) = \sigma (S)$. 
\end{proof}

\section{Main result}

As mentioned in the introduction our main result is the exact value of the Harborth constant for groups of the form $C_3 \oplus C_{3n}$ where $n$ is prime. 
\begin{theorem}
\label{thm_main}
Let $p$ be a prime number. 
Then 
\[
\g(C_3 \oplus C_{3p})  = 
\begin{cases}  
3p + 3 & \text{ for } p \neq 3 \\
3p + 4 & \text{ for } p = 3 
\end{cases}
\]
\end{theorem}

We start by establishing that  those values are lower bounds for the Harborth constant. Then, we establish the existences of the zero-sum subsequences that we need under 
under several additional assumptions on the sequences. Finally, we combine all these results.

\subsection{Lower bounds}
\label{ssec_lower}

In this section we establish lower bounds for the Harborth constant. We start with a general lemma. An interesting aspect of this lemma is that it mixes constants for squarefree sequences and sequences; it improves the result \cite[Lemma 3.2]{mors1}, where instead of the Davenport constant the Olson constant was used. 
\begin{lemma}
\label{lbound_general}
Let $ G_1, G_2 $ be finite abelian groups with $\exp(G_2) \mid \exp(G_1)$. 
Then
\[
\g(G_1 \oplus G_2) \ge  \g(G_1) \oplus  \mathsf{D}(G_2) - 1.
\]
\end{lemma}

\begin{proof}
Let $S_1$ be a squarefree sequence over $G_1$ of length $\g(G_1) - 1$ that has no zero-sum subsequence of length $\exp(G_1)$. 
Let $S_2'$ be a sequence over $G_2$ of length $\mathsf{D}(G_2) - 1$ that has no zero-sum subsequence.
Suppose $S_2' = \prod_{g \in G_2} g^{v_g}$. Since $S_2'$ is zero-sumfree $ v_g < \exp(G_2) \le \exp(G_1)$ for each $g \in G_2$.
Let $\{h_1, \dots, h_{\exp(G_1)-1}\}$ be distinct elements in $G_1$, and let  
\[
S_2 = \prod_{g \in G_2} \left( \prod_{i=1}^{v_g}(g+h_i) \right).
\]
Then, $S_2$ is a squarefree zero-sum free sequence over $G_1 \oplus G_2$.
To show our claim, it suffices to show that  $S_1S_2$ has no zero-sum subsequence of length $\exp(G_1 \oplus G_2)$.
Assume to the contrary that $T \mid S_1 S_2$ is a zero-sum subsequence of length $\exp(G_1 \oplus G_2)$.
Let $T= T_1 T_2$ with $T_i \mid S_i$. 
Since $\exp(G_1 \oplus G_2) = \exp(G_1)$, it follows that $T$ is not a subsequence of $S_1$, that is, $T_2$ is not the empty sequence. 
Let  
\[
\pi \colon 
\begin{cases} 
G & \to G_2 \\
g=g_1 +g_2 & \mapsto g_2 
\end{cases}
\] 
where $g_i\in G_i$ is the unique elements such that $g = g_1 +g_2$, in other words it is the projection on $G_2$.  
Since $\sigma(\pi (T_1))=0$, it follows that $\sigma(\pi (T_2))=0$. 
Yet this is a contradiction, as $\pi (T_2)$ is a non-empty zero-sum subsequence of $S_2'$, 
while by assumption  $S_2'$ has no non-empty zero-sum subsequence. 
\end{proof}

Using this lemma in combination with the result for cyclic groups, yields the following bound, which is given in \cite[Proposition F.102]{bajnok}.
\begin{lemma}
\label{lem_lbsimple}
Let $n_1, n_2$ be strictly positive integers with $n_1 \mid n_2$. Then 
\[
\g(C_{n_1} \oplus C_{n_2})\ge 
\begin{cases}
 n_1+n_2 - 1 &  \text{if  $n_2$ is odd}\\
 n_1+n_2  & \text{if  $n_2$ is even}
\end{cases}.
\]
In particular, 
\[
\g(C_{3} \oplus C_{3n})\ge 
\begin{cases}
 3n+2 &  \text{if  $n$ is odd}\\
 3n+3 & \text{if  $n$ is even}
\end{cases}.
\]
\end{lemma}

\begin{proof}
By Lemma \ref{lbound_general} we have 
$\g(C_{n_1} \oplus C_{n_2}) \ge  \g(C_{n_2}) \oplus  \mathsf{D}(C_{n_1}) - 1$.
The claim follows using that  
\[ 
\g(C_{n_2}) = 
\begin{cases}
 n_2 &  \text{if  $n_2$ is odd}\\
 n_2 + 1 & \text{if  $n_2$ is even}
\end{cases}
\]
and $ \mathsf{D}(C_{n_1})= n_1$ (see, e.g., \cite[Theorem 10.2]{grynkiewicz_book}).  The claim for $C_3 \oplus C_{3n}$ is a direct consequence.
\end{proof}

The bound  for $\g(C_{3} \oplus C_{3n})$ can be improved for odd $n$. 
This was initially done by Kiefer \cite{kiefer} (also see \cite[Proposition F.104]{bajnok}).
We include the argument, as our construction is slightly different.  
\begin{lemma}  
\label{lem_lbimproved}
Let $G=C_{3} \oplus C_{3n}$ with an integer $n \ge 2$.
Then $\g(G) \ge 3n+3$. 
\end{lemma}

\begin{proof}
To prove this lemma, it suffices to give an example of a squarefree sequence of length $3n+2$ that does not admit a zero-sum subsequence of length $\exp(G)=3n$.
Let $G=\langle e_{1} \rangle \oplus \langle e_{2}\rangle $ with $\ord(e_{1})=3$ and $\ord(e_{2})=3n$. 
Let $ \pi' $ and $\pi''$ denote the usual maps $\pi': G \to \langle e_1 \rangle $ and $\pi'': G \to \langle e_2 \rangle$. 

Further, let $T_{1}=  \prod_{g \in \langle e_{2}  \rangle \setminus \{0,-e_{2}, e_{2} \} } (e_{1}+g)$  and 
$T_{2}=  0 (e_{2})(2e_{2})(3e_{2})(-6e_{2})$. 
Then $T = T_1 T_2$ is a squarefree sequence and  $|T|=3n-3+5=3n+2$. 

To obtain the claimed bound, it suffices to assert that $T$ does not have a zero-sum subsequence of length $3n$. 
Assume for a contradiction that $T$ has a zero-sum subsequence $R$ of length $3n$. Clearly, one has $\sigma (\pi'(R))=\sigma (\pi''(R))=0$.
Let $R=R_{1} R_{2}$ with $R_{1}|T_{1}$ and $R_{2}|T_{2}$. Note that $\sigma(\pi'(R))=\sigma(\pi'(R_{1}))=|R_{1}|e_{1}$. 
Consequently, as $\sigma(\pi'(R))=0$ it is necessary that $3$ divides $ |R_{1}|$. 
Moreover to obtain $|R|=3n$ it is necessary that $3n-5 \le |R_{1}| \le 3n-3$. It follows that $|R_{1}|=3n-3$, that is, $R_{1}=T_{1}$. Consequently $|R_2|=3$.

Now, $\sigma(\pi'(R_{1}))=|R_{1}|e_{1}= 0$. Furthermore  
\[
\sigma(\pi''(R_{1})) = \sum_{h \in \langle e_2 \rangle \setminus \{-e_2, e_2, 0\}} h  = \Bigl( \sum_{h \in \langle e_2 \rangle}  h \Bigr)  - (-e_2 + e_2 + 0), 
\]
which is also equal to $0$, since the sum of all elements of the cyclic group $\langle e_{2}\rangle$ is $0$ (here it is used that $3n$ is odd).  
Thus, $\sigma(R_1) = 0$, and it follows that: $\sigma(R) = 0$ if and only if $\sigma (R_2)= 0$.  However,  $T_{2}$ has no subsequence of length $3$ with sum $0$. 
Thus  $T$ has no zero-sum subsequence of length $3n$.
\end{proof}

It turns out that for $n=3$, there is a better construction.

\begin{lemma}  
\label{lb_C3C9}
One has $ \g(C_{3}\oplus C_{9})\geq 13$.
\end{lemma}

\begin{proof}
To prove this lemma, as $\exp(G)=9$, it suffices to give an example of a squarefree sequence $T$ of length $12$ over $G$ that does not admit any zero-sum subsequence $T_{1}$ of length $9$.
Let $G=\langle e_{1} \rangle \oplus \langle e_{2}\rangle $ with $\ord(e_{1})=3$, and $\ord(e_{2})=9$. 

Let us consider the following sequence: 
\[
T=R (e_{1}+R) (e_{2}+R) (e_{1}+e_{2}+R), \text{ with } R=0 (3e_{2}) (6e_{2}),
\]
This is a squarefree sequence of length $12$ that satisfies $\sigma(T)=0+0+3e_{2}+3e_{2}=6e_{2}$.  By Lemma \ref{complement} with $k=12$ and $r=9$, the sequence $T$ has a zero-sum subsequence of length $9$ if and only if $T$ has a subsequence $T_2$ with $|T_2|= 3 = 12-9$ and $\sigma(T)=\sigma(T_2)=6e_2$.
For a contradiction let us assume such a subsequence $T_2$ exists.  Let $H=\{0,3e_{2},6e_{2}\}$ and let $\pi : G \to G/H$ be the standard epimorphism. 
One has $G/H \cong C_{3}\oplus C_{3}$, and this group is generated by $f_{1}=\pi(e_{1})$ and $f_{2}=\pi{(e_{2})}$. 

Since $\sigma(T_2)=6e_2$, one has that $\pi(T_2)$ is a zero-sum subsequence of $\pi(T)$ and $\pi(\sigma(T))=\pi(6e_{2})=0$.

But, note that the only subsequences of  $\pi(T)=0^{3} {f_{1}}^{3} {f_{2}}^{3} {(f_{1}+f_{2})}^{3}$ of length $3 $ which have sum zero are $0^{3}$, ${f_{1}}^{3}$, ${f_{2}}^{3}$ and  ${(f_{1}+f_{2})}^{3}$. 
It remains to check if any of the corresponding subsequences of $T$ has sum $6e_2$. This is not the case. Concretely, we have $\sigma(0^{3})=0 $,
$\sigma(e_{1}^{3})=0 $,
$\sigma(e_{2}^{3})=3 e_{2} $, and  $\sigma((e_{1}+e_{2})^{3})=3e_{2} $.
Thus, the sequence $T$ does not have any subsequence of length $3$ with sum $6e_2$. This establishes the claimed bound. 
\end{proof}

\subsection{Establishing the existence of zero-sum subsequence of length $\exp(G)$ under various assumptions}

Let us fix some notation that will be used throughout the subsection. Let $G=C_3\oplus C_{3p}$ with $p \ge 5$ a prime number. We note that $G = H_1 \oplus H_2$ where $H_1\cong C_3^2$ is the subgroup of elements of order dividing $3$ and $H_2 \cong C_p$ is the subgroup of elements of order dividing $p$. 

For $i \in \{1,2\}$, let  
\[\pi_i \colon 
\begin{cases} G & \to H_i \\
g=h_1 +h_2 & \mapsto h_i \end{cases}\] 
where $h_i\in H_i$ is the unique element such that $g = h_1 +h_2$. That is, $\pi_i$ denotes the projection on the subgroup $H_i$. 

For a sequence $S$ over $G$ there exists a unique decomposition $S = \prod_{h \in H_1}S_h$ where $S_h$ is the subsequence of elements of $S$ with $\pi_1(g)=h$.
If $S$ is squarefree then for each $h \in H_1$ the sequence $\pi_2 (S_h)$ is a squarefree sequence over $H_2$.

To establish the bound $\g(G)\le 3p+3$ we need to show that every squarefree sequence of length $3p + 3$ over $G$ has a zero-sum subsequence of length $3p$. 
By Lemma \ref{complement} this is equivalent to establishing that every squarefree sequence of length $3p + 3$ over $G$ has a subsequence $R$  of length $3$ with the same sum as $S$.

To obtain such a sequence of length $3$ we typically first restrict our considerations  to subsequence for which $\pi_1(\sigma(S))=  \pi_1(\sigma (R))$; this condition can be established via explicit arguments, as the group $H_1$ is fixed and small. 
Then, using tools from Additive Combinatorics recalled in Section \ref{sec_prel}, we show that among the sequences with $\pi_1(\sigma(S))=  \pi_1(\sigma (R))$ there is one for which we also have $\pi_2(\sigma(S))=  \pi_2(\sigma (R))$ and thus satisfy $\sigma(S)= \sigma(R)$ as needed.

We formulate a technical lemma that is a key-tool in our argument. Note that for the proof of this lemma it is crucial that $p$ is prime.

\begin{lemma}
\label{lemma_main} 
Let $S$ be a sequence of length $3p+3$ over $G$. Let $S= \prod_{h \in H_1}S_h$ where $S_h$ is the subsequence of elements of $S$ with $\pi_1(g)=h$.
\begin{enumerate}
\item If there exist  distinct $x,y,z \in H_1$ with $x+y + z =  \pi_1 (\sigma(S))$ such that $S_x,  S_y, S_z$ are all non-empty and $|S_x| + |S_y| + |S_z| -2 \ge p$, then $S$ has a zero-sum subsequence of length $3p$.
\item If there exist distinct $x,y \in H_1$  with $2x + y =  \pi_1 (\sigma(S))$ such that $|S_x| \ge 2$  and $|S_y|\ge 1$ and $2|S_x| + |S_y|-4  \ge p$, then $S$ has a zero-sum subsequence of length $3p$.
\item If there exist  $x \in H_1$ with $3x  =  \pi_1 (\sigma(S))$ such that $|S_x| \ge 3$  and
$3|S_x|  - 8\ge p $, then $S$ has a zero-sum subsequence of length $3p$.
\end{enumerate}
\end{lemma}

We use and combine the theorems of Cauchy--Davenport and Dias da Silva--Hamidoune. 

\begin{proof}
In each case we show that under the assumptions of the lemma, $S$ has a subsequence $R$ of length $3$ with the same sum. By Lemma \ref{complement} with $k=3p+3$ and $r=3$ this establishes our claim.
 
\noindent 
(1). Let $x,y,z \in H_1$ be distinct with $x+y + z =  \pi_1 (\sigma(S))$. 
If $g_x$ divides $S_x$ and $g_y$ divides $S_y$ and $g_z$ divides $S_z$, then $g_xg_yg_z$ is a subsequence of $S$ and $\pi_1 (\sigma(g_xg_yg_z)) = x+y+z =  \pi_1 (\sigma(S))$.  

Thus, to show that $S$ has a subsequence $R$ of length $3$ it suffices to show that there exist elements $g_x$, $g_y$ and $g_z$ such that   
$g_x$ divides $ S_x$ and $g_y $ divides $ S_y$ and  $g_z $ divides $ S_z$ with $\pi_2 (\sigma(g_xg_yg_z)) =  \pi_2 (\sigma(S))$. 

Let $\Omega$ denote the set of all sequence $g_xg_yg_z$ with $g_x\mid S_x$, $g_y \mid S_y$, $g_z \mid S_z$. 
We note that 
\[\bigl\{ \pi_2 (\sigma(R)) \colon  R \in \Omega \bigr\} = \supp ( \pi_2 (S_x)) + \supp ( \pi_2 (S_y)) + \supp  (\pi_2 (S_z)).
\]

From Theorem \ref{CDthm}, the Theorem of Cauchy--Davenport,  
\begin{align*}
&   |\supp ( \pi_2 (S_x)) +  \supp ( \pi_2 (S_y)) + \supp  (\pi_2 (S_z))|   \ge \\ 
&  \min \bigl\{p , |\supp ( \pi_2 (S_x))| + |\supp ( \pi_2 (S_y))| + |\supp  (\pi_2 (S_z))|- 2 \bigr\} 
\end{align*}
As the sequence $S$ is squarefree, for each $h \in H_1$, the sequence $\pi_2(S_h)$ is squarefree as well.  Consequently, $|\supp ( \pi_2 (S_h))|= |S_h|$. 
Thus, if $|S_x| + |S_y| + |S_z| -2 \ge p$, then $\supp(\pi_2(S_x)) + \supp(\pi_2(S_y)) + \supp(\pi_2(S_z))$ must be equal to the full group $H_2$.
In particular, there exits a sequence $R \in \Omega$ with $\pi_2(\sigma(R) ) = \pi_2 (\sigma(S))$, and the proof is complete.

\medskip
\noindent
(2). Let $x,y \in H_1$ be distinct with $2x + y  =  \pi_1 (\sigma(S))$. 
If $g_xg'_x\mid S_x$ and  $g_y \mid S_y$, 
then $g_xg'_xg_y$ is a subsequence of $S$ and $\pi_1 (\sigma(g_xg_x'g_y)) = 2x+y =  \pi_1 (\sigma(S))$.  

Let $\Omega$ denote the set of all sequence $g_xg_x'g_y$ with $g_xg_x'$ divides $S_x$ and  $g_y$ and $S_y$. 
We note that $\{ \pi_2 (\sigma(R)) \colon  R \in \Omega \} = \Sigma_2( \supp(\pi_2(S_x))) + \supp(\pi_2(S_y))$. 

By the Theorems of Dias da Silva--Hamidoune and Cauchy--Davenport (see Theorems \ref{thmDH} et \ref{CDthm}) we get that, as $p$ is assumed to be prime, 
\[\bigl|\Sigma_2 \bigl(\supp(\pi_2(S_x))\bigr) + \supp\bigl(\pi_2(S_y)\bigr) \bigr| \ge \min \bigl\{p, 2   |\supp ( \pi_2 (S_x))| + |\supp ( \pi_2 (S_y))| - 4 \bigr\}.\]
As in (1), if $2|S_x| + |S_y| -  4 \ge p$, then there exists  some $R \in \Omega$ with $\pi_2(\sigma(R) ) = \pi_2 (\sigma(S))$, and the proof is complete.

\medskip
\noindent
(3). Let $x \in H_1$ be with $3x = \pi_1 (\sigma(S))$. 
Let $g_xg'_xg''_x$ be a subsequence of $S_x$ of length $3$. Then $g_xg'_xg_x''$ is a subsequence of $S$ with $\pi_1 (\sigma(g_xg_x'g_x'')) = 3x = \pi_1 (\sigma(S))$.  

Let $\Omega$ denote the set of all subsequence $g_xg_x'g_x''$ of $S_x$. 
We note that 
\[
\bigl\{ \pi_2 (\sigma(R)) \colon  R \in \Omega \bigr\} = \Sigma_3 \bigl( \supp(\pi_2(S_x)) \bigr).
\] 

Similarly, by Theorem \ref{CDthm}, the Theorem of Dias da Silva--Hamidoune, we get: 
\[
\bigl|\Sigma_3( \supp(\pi_2(S_x)))\bigr| \ge \min \bigl\{p, 3   |\supp ( \pi_2 (S_x))|  - 8 \bigr\}.
\]
As in (1), if $3|S_x| -  8 \ge p$, then there exists  some $R \in \Omega$ with $\pi_2(\sigma(R) ) = \pi_2 (\sigma(S))$, and the proof is complete.
\end{proof}

In the present context there exists essentially two types of sequences $S$ over $G$ of length $3p+3$: those for which $\pi_1(\sigma(S))$ equals zero and those for which it is non-zero.
It is clear that this property is preserved under automorphisms of the group, and when the length of the sequence is a multiple of $3$ it is also preserved under translations.  
We treat these two types of sequences separately. The latter type is treated in Proposition \ref{prop_nonzero}. For the former type,  two cases are distinguished: the case where  
the support of $\pi_1 (S)$ is the full group $H_1$ (see Proposition \ref{prop_zero_full}) and the cases where it is not (see Proposition \ref{prop_zero}).

\begin{proposition}
\label{prop_zero_full}
Let $S$ be a squarefree sequence over $G$ of length $3p+3$. If $\sigma ( \pi_1 (S)) = 0$ and $\supp (\pi_1 (S))=H_1$, then $S$ has a zero-sum subsequence of length $3p$.
\end{proposition}
\begin{proof}
To simplify the subsequent considerations, we note that we can assume that $\sigma(\pi_2(S))=0$ (so effectively $\sigma(S)=0$). 
Indeed, it suffices to note that  if, for any $h \in G$, the shifted sequence $h + S$ contains a zero-sum subsequence of length $3p$, then the sequence $S$ contains a zero-sum subsequence of length $3p$. There is some $h' \in H_2$ such that $(3p+3)h' = -\sigma(\pi_2(S))$; note that as $p$ and $3p+3$ are co-prime, the multiplication $h \mapsto (3p+3)h$ is an isomorphism on $H_2$.  Now, one can consider $h'+S$ instead of $S$ provided the additional condition $\sigma(\pi_{1}(S))=0$ is not altered. Since $\sigma ( \pi_1 (h'+S)) = |S|h'+\sigma(\pi_{1}(S))$ and since $|S|h'=(3p+3)h'=0$, this is indeed true and $\supp (\pi_1 (h'+S))=H_1$ still holds.   

Let $H_1'$ be a cyclic subgroup of $H_1$ and let $g \in H_1$, and $\{x, y, z\}=g+H_1'$ a co-set. Since $x+y+z= 0 = \sigma (\pi_1 (S))$, it follows that if $|S_xS_yS_z| -2 \ge p$, then from part (1) of Lemma \ref{lemma_main}, the result holds.

It remains to consider the case where for each co-set of $H_1$ of cardinality three,  denoted $\{x, y, z\}$, one has $|S_xS_yS_z| \leq p+1$. 
We note that this is only possible if for every co-set $\{x,y,z\}$, one has $|S_xS_yS_z| =  p+1$. Indeed,  $H_1$ can be partitioned as the disjoint union of three such co-sets, 
say, $H_1 = \{x_1,y_1,z_1\} \cup \{x_2,y_2,z_2\} \cup \{x_3,y_3,z_3\}$. Then, on the one hand 
\[|S_{x_1}S_{y_1}S_{z_1}| \leq p+1,  \quad |S_{x_2}S_{y_2}S_{z_2}| \leq p+1, \quad | S_{x_3}S_{y_3}S_{z_3}| \leq p+1,\] yet on the other hand $|S_{x_1}S_{y_1}S_{z_1}| + |S_{x_2}S_{y_2}S_{z_2}| +| S_{x_3}S_{y_3}S_{z_3}| = |S|= 3p+3$. Therefore it is necessary that for each $i \in [1,3]$ one has$|S_{x_i}S_{y_i}S_{z_i}| =p+1$.  

Next we assert that this is only possible if each of the $9$ sequences has the same length. 
Let $H_1 = \{ x_1', x_2', \cdots , x_9' \}$ such that $|S_{x_1'}| \geq |S_{x_2'}| \geq \cdots \geq |S_{x_9'}|$. Let $v_i = |S_{x_i'}|$. Let $j\in [1,9]$ such that $\{x_1', x_2', x_j'\}$ is a co-set and let $i\in [1,9]$ such that $\{x_8', x_9', x_i'\}$ is a co-set.   
Then one has $v_1 + v_2 + v_j =p+1$  and $v_8 +v _9 + v_i=p+1$. 

It follows that $ (v_1-v_9) + (v_2-v_8) + (v_j-v_i)=0$. 
Hence $ (v_1-v_9) +(v_2-v_8)=v_i-v_j$. Yet, as $(v_1-v_9) \geq (v_i-v_j)$ we get  $v_2 -v_8=0$; notice that $ v_2-v_8 \geq 0$, if $v_2-v_8 >0$, hence $(v_1-v_9)+ (v_2-v_8) \geq (v_i-v_j) + (v_2-v_8) > v_i-v_j$, a  contradiction.
Consequently, one has  $v_2 = v_3 = \dots  = v_8 = v$, and this common value is $\frac{p+1}{3}$. Note that if this is not an integer, the proof is complete. So it can be assumed that $\frac{p+1}{3}$ is an integer. It remains to show that $v_1 = v$ and $v_9 = v$. There exists a co-set of $H_1$ of cardinality $3$ that contains $x_9'$  and that does not contain $x_1'$, so $v_9  + 2v = p+1$ and thus  $v_9= \frac{p+1}{3}$. In the same way we get that $v_1 = v$.

We now reconsider, for a co-set $\{x, y, z\}$, the cardinality of the set   $ \supp( \pi_2 (S_x)) + \supp (\pi_2 (S_y)) + \supp (\pi_2 (S_z))$.
By Theorem \ref{CDthm}, the Theorem of Cauchy--Davenport, one has the lower bound  $\min \{p,|\supp( \pi_2 (S_x))| + | \supp(\pi_2 (S_y))| + |\supp(\pi_2 (S_z))| -2 \} = \min\{p, |S_x| + | S_y| +|S_z| -2\}= p-1$. If one has $| \supp (\pi_2 (S_x)) + \supp( \pi_2 (S_y)) + \supp (\pi_2 (S_z))| \ge p$, then  $\supp (\pi_2 (S_x)) + \supp (\pi_2 (S_y)) + \supp (\pi_2 (S_z)) = H_2$, and we can conclude as previously.

Thus, it remains to consider the case that $|\supp( \pi_2 (S_x)) + \supp (\pi_2 (S_y)) + \supp (\pi_2 (S_z))|=p-1$.

Now, this is only possible when
\begin{align*}
& | \supp (\pi_2 (S_x)) + \supp (\pi_2 (S_y)) + \supp (\pi_2 (S_z))| = \\ 
& |\supp (\pi_2 (S_x)) + \supp( \pi_2 (S_y))| + | \supp (\pi_2 (S_z))| -1 =  \\ 
& |\supp(\pi_2 (S_x))| + |\supp (\pi_2 (S_y))|-1 + |\supp (\pi_2 (S_z))| -1.
\end{align*}

Thus, $|\supp (\pi_2 (S_x)) + \supp (\pi_2 (S_y))| = |\supp (\pi_2 (S_x))| + |\supp(\pi_2 (S_y))|-1$. 
In the same way, it follows that $|\supp (\pi_2 (S_x)) + \supp (\pi_2 (S_z))| = |\supp (\pi_2 (S_x))| + |\supp(\pi_2 (S_z))|-1$, and  that 
$|\supp (\pi_2 (S_y)) + \supp (\pi_2 (S_z))| = |\supp (\pi_2 (S_y))| + |\supp(\pi_2 (S_z))|-1$.
 
The Theorem of Vosper, see Theorem \ref{vosper}, yields that $\pi_2 (S_x),\pi_2 (S_y)$  are arithmetic progressions with the same common difference, 
$\pi_2 (S_x),\pi_2 (S_z)$  are arithmetic progressions with the same common difference, and $\pi_2 (S_y),\pi_2 (S_z)$  are arithmetic progressions with the same common difference. 
From Lemma  \ref{lem_diff_ap}, it follows that there is a common difference for all three  $\pi_2(S_x), \pi_2 (S_y),\pi_2 (S_z)$.
Indeed, one has this for any pair of the $9$ sets, as the argument can be applied for any co-set. Thus, all $9$ sets are arithmetic progressions with a common difference. 
Let us denote this difference by $e$; of course, this is a generating element of $H_1$. 

If there is some $h \in H_1$ such that $\pi_2 (S_h)$ has a zero-sum subsequence of length $3$, then in fact $S_h$ has a zero-sum sequence of length $3$. Since we assumed at the start that  $\sigma(S)=0$, invoking Lemma \ref{complement} our claim is complete.  

Thus, we assume that for no $h\in H_1$ the sequence $\pi_2 (S_h)$ has a zero-sum subsequence of length $3$. In particular $\pi_2 (S_h)$ does not have $(-e)e0$, as a subsequence. Thus, for each $h \in H_1$ one has
$\pi_2 (S_h) = \prod_{j=s_h}^{s_h+{\frac{p-2}{3}}}(j e)$ with $0\le |S_h| \le |S_h| + \frac{p-2}{3} < p-1$.
It is easy to see that: 
\[\Sigma_3 (\pi_2 (S_h)) = \bigl\{ j e \colon j \in [3 |S_h|+3,3|S_h| - 3 + (p-2)] \bigr\}.\]

For this set not to contain $0$, we need $3|S_h| - 3 + (p-2)< p$. So $3 |S_h|  < 5$, that is $|S_h| \in \{0,1\}$.

If there is a co-set $\{x,y,z\}$ such that $|S_x| = |S_y| = |S_z| = 0$, then clearly $\pi_2 (S_x) + \pi_2 (S_y) + \pi_2 (S_z)$ contains $0$.
Yet if there is no co-set $\{x,y,z\}$ such that $|S_x| = |S_y| = |S_z| = 0$, then there is a co-set $\{x',y',z'\}$ such that $|S_{x'}| + |S_{y'}| + |S_{z'}|  \ge 2$;
indeed, it suffices to note that by the former condition there must be at least two elements $h,h' \in H_1$ with $|S_h| \ge 1$ and $|S_{h'}| \ge 1$. However, this gives that $\pi_2 (S_{x'}) + \pi_2 (S_{y'}) + \pi_2 (S_{z'})$ will contain $(2 \cdot \frac{p+1}{3} + \frac{p-2}{3})e = p e = 0$. 
Thus, the argument is complete.
\end{proof}

For the next result, we  keep the condition that $\sigma ( \pi_1 (S)) = 0$, yet consider the case $\supp (\pi_1 (S))\neq H_1$ instead. 

\begin{proposition}
\label{prop_zero}
Let $S$ be a squarefree sequence over $G$ of length $3p+3$. If $\sigma ( \pi_1 (S)) = 0$ and $\supp (\pi_1 (S))\neq H_1$, then $S$ has a zero-sum subsequence of length $3p$.
\end{proposition}

\begin{proof}
Let $h \in H_1$ such that $|S_h|=0$; such an element exists by assumption. 
Now, as recalled in Section \ref{sec_prel}, the sequence $S$ contains a zero-sum subsequence $T$ of length $3p$ then the sequence $S-h$ contains $T-h$ as a zero-sum subsequence of length $3p$.   Since $\supp (\pi_1 (-h+S))= -h + \supp (\pi_1 (S))$ it follows from $h \notin \supp (\pi_1 (S))$  that $0 \notin \supp (\pi_1 (-h+S))$.  
Thus, by translation, it can be assumed without loss of generality that $|S_0|=0$.

Thus, we can consider $-h+S$ instead of $S$ provided the additional condition $\sigma ( \pi_1 (S)) = 0$ is not altered.  Since $\sigma ( \pi_1 (-h+S)) = |S|(-h)+ \sigma ( \pi_1 (S))$ and since  $|S|h= (3p+3)h=0$, this is indeed true.     

Now, we distinguish cases according to the cardinality of $|\supp(\pi_1(S))|$. By assumption it is strictly less than $|H_1|=9$. 

Suppose that $|\supp(\pi_1(S))|=8$.  We note that there exists exactly $8$ co-sets of cardinality $3$ that do not contain $0$. Each non-zero element is contained in exactly $3$ of them. Thus there exists as co-set $\{x,y,z\}$  such that $|S_xS_yS_z | \ge \frac{3}{8}|S| = \frac{9p+9}{8}$. The existence of the required subsequence now follows from part (1) of Lemma \ref{lemma_main} as $(9p+9)/8 > p+1$ for $p \ge 5$.

Suppose that $|\supp(\pi_1(S))|=7$. Let $-x \in H_1$ be the non-zero element such that $|S_{-x}| = 0$. We note that there are $4$ co-sets of cardinality $3$ that contain $x$, and $3$ of those contain neither $-x$ nor $0$. 
It thus follows that there exists a co-set  $\{x,y,z\}$ such that 
such that $|S_xS_yS_z | \ge |S_x| + \frac{1}{3}(|S_x^{-1}S|) = (p+1) +\frac{2|S_x|}{3}> p+1$. Using part (1) of Lemma \ref{lemma_main} again, the existence of the required subsequence follows. 

Suppose that $|\supp(\pi_1(S))|=6$. Let $g,h \in H_1$ be the two non-zero elements such that $|S_g| = |S_h| = 0$. If $g= -h$, then there is a co-set $\{x,y,z\}$ with respect to the subgroup $\{0,g,-g\}$ such that $|S_xS_yS_z | \ge \frac{1}{2}|S| = \frac{3p+3}{2}> p+1$ (note that the two co-sets other than $\{0,g,-g\}$ itself cover the $6$ remaining elements of $H_1$). Again, from part (1) of Lemma \ref{lemma_main} the argument is complete. 

If $g \neq -h$, then for one of the two co-sets $\{-g, -g +h, -g -h\}$ and 
$\{-h, -h +g,  -h - g\}$ one has $|S_xS_yS_z | \ge  |S_{-g-h}|+  \frac{1}{2} |S_{g+h}S| =  (3p+3 -|S_{g+h}| + |S_{-g-h}|)/2$; note that both co-sets contains $-g-h$ and union  of the two co-sets contains all element of $H_1 \setminus \{0,g,h\}$ except for $g+h$.  Since $|S_{g+h}|\le p$, it follows that $|S_xS_yS_z |>p+1$ and again from part (1) of Lemma \ref{lemma_main} the argument is complete.  

Suppose that $|\supp ( \pi_1(S))|\le 5$. In this case there exists some $h \in H_1$ such that  $|S_h| \ge \frac{1}{5}|S| = \frac{3p+3}{5}$. If $p> 5$, by applying part (3) of Lemma \ref{lemma_main} to $S_{h}$, then we can complete the proof; for $p \ge 11$ this is direct and for $p=7$ we observe that one has $|S_h| \ge 5$. 
It remains to consider the special case $p=5$. 
The part (3) of Lemma \ref{lemma_main} can be applied if there exists some $h \in H_1$ with $|S_h| = 5$. Thus assume that $|S_h|\le 4$ for all $h \in H_1$. This implies $|\supp (\pi_1(S))|= 5$ since otherwise there would exist some $h \in H_1$ with  $|S_h| \ge 18/4 > 4$. 

Let $\{h_1, \dots, h_5\} \subset H_1$ such that $|S_{h_i}|\neq 0$ for each $i \in [1,5]$. 
Since $\mathsf{g}(C_3^2)= 5$, as recalled in the Introduction, there exist distinct $i,j,k \in [1,5]$ such that $h_i + h_j +h_k = 0$. 
This is equivalent to $\{h_i , h_j, h_k\} \subset H_1$ being a co-set. Now, $|S_{h_i}S_{h_j}S_{h_k}| = |S| - 2  \max\{|S_{h}| \colon h \in H_1\} \ge 18 - 2 \cdot 4 = 10$. Again we can apply part (1) of  Lemma \ref{lemma_main} to complete the argument.  
\end{proof}

\begin{remark}
For $p$ sufficiently large  a shorter argument is available.  
There exists some $h\in H_1$ such that  $|S_h|\ge \frac{1}{8} |S| =  \frac{3p+3}{8}$.
We can apply part (3) of Lemma \ref{lemma_main} if  $3 |S_h|-8 \ge p$. 
This is true provided that 
 $3 \cdot \frac{3p+3}{8}-8 \ge p$, which is equivalent to  $\frac{p}{8} - \frac{55}{8} \geq 0$. Hence for $p \geq 55$ we can complete the argument in this way. 
\end{remark}

We now turn to the case $\sigma ( \pi_1 (S)) \neq 0$. 

\begin{proposition}
\label{prop_nonzero}
Let $S$ be a squarefree sequence over $G$ of length $3p+3$. If $\sigma ( \pi_1 (S)) \neq 0$, then $S$ has a zero-sum subsequence of length $3p$.
\end{proposition}

\begin{proof}
Let $c = \sigma(\pi_1(S))$. 
Without loss of generality, one can assume that $|S_c|$ is maximal among all $|S_h|$ for $h \in H_1$. The argument is the same as in the proof of Proposition \ref{prop_zero}.

Notice that $\frac {|S|}{9}=\frac{3p+3}{9}=\frac{p+1}{3}$ and thus $|S_c| \ge \frac{p+1}{3}$. 
Let us show that $|S_{c}| > \frac{p+1}{3}$.  If for each $h \in H_{1}$ one has $|S_h| = \frac {p+1}{3}$,
then $\sigma(\pi_1(S)) = \frac{p+1}{3} \sum_{h \in H_1}h$. Yet $\sum_{h \in H_1}h = 0$, which contradicts  $\sigma(\pi_1(S))= c \neq 0$.
Thus $|S_h| \neq \frac {p+1}{3}$  for some $h \in H_1$, and thus $|S_c| > \frac{p+1}{3}$. 

The strategy of the proof is again to apply Lemma \ref{lemma_main}. 
To this end we need to find subsequence of $\pi_1(S)$ of length $3$ that have sum $c$. 

One possibility is to consider such subsequence formed by elements from the the cyclic subgroup of $C= \{-c, 0, c\}$ only. Thus the  subsequences of this subgroup of length $3$ which have sum $c$ are: $0^{2}c$ and $ (-c)^{2}0$ and $c^{2}(-c)$. This approach works if sufficiently many elements from the sequence $S$ are contained in this subgroup. 
This is detailed in case 1 below.

Another possibility is to consider subsequences of the form $c h (-h)$ with $h \notin C$. 
While not phrased explicitly in this form, the distinction of subcases in case 2 corresponds to the number (counted without multiplicity) 
of distinct subsequences  of this form in  $\pi_1(S)$. 

Let $v_{C}=|S_{0}S_{c}S_{-c}|$.

\noindent
\textbf{Case 1:} $v_{C} \ge p+3$. If $|S_{-c}| = 0$, then  $|S_{0}|+|S_{c}|=v_{C} \ge p+3$. Thus, $|S_{0}| \ge 3$ and $|S_{c}| \ge 1$. 
Thus $|S_c| +2  |S_0|=|S_{c}|+|S_{0}|+|S_{0}| \geq p+3+2 = p + 5$. From part (2) of Lemma \ref{lemma_main} applied with $x=0$ and $y=c$, the claim follows.

If $|S_0| \leq  1$, then $|S_{-c}| \ge 2$ and thus $ 2 \cdot |S_c| +  |S_{-c}| \geq  v_C - 1 + |S_c|  \geq p    + 4$.
The claim follows from part (2) of Lemma \ref{lemma_main}, applied with $x=c$ and $y=-c$.

If $|S_{-c}| \ge 1$ and $|S_0| \geq  2$, and one of  $|S_c| +2  |S_0| \geq p  + 4$ or $ 2  |S_c| + |S_{-c}| \geq  p + 4$ are true, then the claim follows from part (2) of Lemma \ref{lemma_main}; notice that $c + 2 \cdot 0 = c $ and $ 2 c + (-c)=c $. 
Thus, assume $|S_c| +2  |S_0| \le p  + 3$ and $ 2  |S_c| +  |S_{-c}| \leq  p   + 3$.

Summing the two inequalities, it follows that  $3  |S_c| + 2 |S_0| + |S_{-c}| \leq 2p+6$. Since $v_C =  |S_{c}| + |S_{0}| + |S_{-c}| \ge p+3$, it follows that  $v_C = p + 3$ and $|S_{c}| = |S_{-c}|$. Since $|S_{c}| \ge \frac {v_C}{3} = \frac {p+3}{3}$, which is not an integer, it follows that in fact $|S_{c}| \ge \frac {p+4}{3}$.
Yet, then $ 2  |S_{c}| +  |S_{-c}| =3   |S_{c}|  \geq p    + 4$, and the claim follows again.

\smallskip
\noindent
\textbf{Case 2:} $v_C \le p+2$. The set $H_1 \setminus C$, can be partitioned into three subsets of size two, each containing an element and its inverse, say $H_1 \setminus C =\{ g_1, -g_1, g_2, -g_2, g_3  , -g_3\}$. Possibly exchanging the role of $g_i$ and $-g_i$, one can assume that for each $i \in [1,3]$ one has $|S_{g_i}| \ge |S_{-g_i}|$. 
In addition, by renumbering if necessary, one can assume that $|S_{-g_1}| \ge |S_{-g_2}| \ge |S_{-g_3}|$. Adopting this convention we get that $|S_{-g_3}|>0$ implies that in fact all six sequence $S_h$ for $h \in H_1 \setminus C$ are non-empty. However, note that we do not know if, say, $|S_{g_1}|\ge |S_{g_2}|$; we only know $|S_{g_1}|\ge |S_{-g_1}| \ge |S_{-g_2}|$ and $|S_{g_2}|  \ge |S_{-g_2}|$.    

\noindent
\textbf{Case 2.1:}  $|S_{-g_3}| > 0$. Let $i \in \{1,2,3\}$ such that  $|S_{g_i}S_{-g_i}|$ is maximal among $|S_{g_1}S_{-g_1}|$, $|S_{g_2}S_{-g_2}|$, and $|S_{g_3}S_{-g_3}|$. Thus  $|S_{g_i}S_{-g_i}| \geq \frac{3p+3 - v_C}{3}$.

Hence 
\[
|S_{c}| + |S_{g_i}| + |S_{- g_i}| \ge  \frac{3p+3 - v_C}{3}  +  |S_{c}|  = (p+1) +  \left(|S_{c}| -  \frac{v_C}{3}\right).\]
Thus, $|S_{c}| + |S_{g_i}| + |S_{-g_i}| \geq p+1$ with equality if and only if $|S_{c}| =   \frac{v_C}{3}$ and $|S_{g_i}| + |S_{-g_i}|=  \frac{3p+3- v_C}{3}$.  
If equality does not hold, then the claim follows from part (1) of  Lemma \ref{lemma_main} as one has $|S_{c}| + |S_{g_1}| + |S_{-g_1}|  > p +1$.

Thus assume that one has equality, that is, assume $|S_{c}| =   \frac{v_C}{3}$ and $|S_{g_i}| + |S_{-g_i}|=  \frac{3p+3- v_C}{3}$.

The former implies that  $|S_{c}|=|S_{-c}|=|S_{0}|$ and since $|S_{c}| \ge \frac{p+2}{3}$ (recall the argument at the very beginning of the proof) while  $v_C \le p+2$ (this is the assumption of Case 2) we get that in fact $v_C = p+2$, and thus $ \frac {|S_{c}|}{3} = \frac{p+2}{3}$. Furthermore, we can now infer that $|S_{g_i}| + |S_{-g_i}|= \frac{2p+1}{3}$.  Yet since   $|S_{g_i}|\le |S_c|$, this is only possible if $|S_{g_i}|=\frac{p+2}{3}$ and $|S_{-g_i}|= \frac{p-1}{3}$.  The same holds true for each of  $g_1,g_2, g_3$. Therefore,  one has 
\[
c=\sigma(\pi_1 (S))=\frac {p+2}{3}  \bigl(c+(-c)+0 + g_1+g_2+g_3 \bigr)+ \frac{p-1}{3}  (-g_1-g_2-g_3)=g_1+g_2+g_3.\]
Now, we can apply part (1) of  Lemma \ref{lemma_main} with $g_1, g_2, g_3$; note that  $|S_{g_1}|+ |S_{g_2}|+|S_{g_3}| = 3 \cdot \frac{p+2}{3} =p+2$. (In fact it can be seen that  $g_1 + g_2 + g_3 = c$ is impossible. To assert this would be another way to conclude the argument.)

\noindent
\textbf{Case 2.2:}  $|S_{-g_2}| > 0$ and $|S_{-g_3}| = 0$.
One has $|S_{g_1}S_{-g_1}S_{g_2}S_{-g_2}|   = 3p+3  - v_C - |S_{g_3}|  \ge  3p+3  - v_C - |S_{c}|$.
Let $i \in \{1,2\}$ such that $|S_{g_i}S_{-g_i}|$ is maximal among  $|S_{g_1}S_{-g_1}|$ and $|S_{g_2}S_{-g_2}|$.
Then,
\[
|S_{c}| + |S_{g_i}| + |S_{- g_i}| \ge  \frac{3p+3 - v_C - |S_{c}|}{2}  + |S_{c}| =
p+1 + \frac{p+1 - v_C +  |S_{c}|}{2}.
\] 
Now, since $v_C \le p+2$  and $|S_{c}| \ge \frac{p+2}{3} \ge 2$, it follows that $|S_{c}| + |S_{g_i}| + |S_{- g_i}| \ge p+2$ and one can apply part (1) of Lemma \ref{lemma_main} with $c,g_i, -g_i$.

\noindent
\textbf{Case 2.3:}  $|S_{-g_1}| > 0$ and $|S_{-g_2}| = |S_{-g_3}| = 0$.
One has $|S_{g_1}S_{-g_1}|   = 3p+3  - v_C - |S_{g_2}S_{g_3}|  \ge  3p+3  - v_C -2 \cdot  |S_{c}|$.
Thus, $|S_{c}| + |S_{g_1}| + |S_{- g_1}| \ge  3p+3 - v_C  - |S_{c}|$. Since $v_C \le p +2$  and $|S_{c}|\le p $, it follow that   $|S_{c}| + |S_{g_1}| + |S_{- g_1}| \ge p$ 
with equality if and only if  $v_C = p +2$  and $|S_{c}|= p$.
If equality does not hold, the claim follows from  part (1) of Lemma \ref{lemma_main} with $c,g_1,-g_1$. 
Thus, we assume $v_C = p+2$ and $|S_{c}|= p$.
If $|S_{-c}| \neq 0$, then the claim follows from part (2) of Lemma \ref{lemma_main} with $x=c$ and $y = -c$, as $2  |S_{c}| + |S_{-c}| \ge p+4$.  
If $|S_{-c}| = 0$, then $|S_{0}| = v_C - |S_c| = 2$ and the claim follows from part (2) of Lemma \ref{lemma_main} with $x=0$ and $y=c$ as $ 2  |S_{0}| + |S_c|= p+4$ and we are done again.

\noindent
\textbf{Case 2.4:}  $|S_{-g_1}|=|S_{-g_2}| = |S_{-g_3}| = 0$. One has $|S_{g_1}S_{g_2}S_{g_3}|  = 3p + 3 -v_C  \ge 2p+1$
If $|S_{c}|= p$, then we can assume $v_C \le p+1$ (see the argument at the end of the preceding case). Thus, in this case
  $|S_{g_1}S_{g_2}S_{g_3}|  \ge 2p+2$.
It follows that for each $i \in \{1, 2, 3 \}$, one has $|S_{g_i}| \ge 2$, and thus
$2|S_{g_i}|+ |S_{g_j}| \ge 2 + (2p+2 -p)= p+4$, for each choice of distinct $i,j \in \{1,2,3\}$.

If $|S_{c}|\le  p-1$, then it follows that for each $i \in \{1, 2, 3 \}$, one has$|S_{g_i}| \ge 2p+1 - 2(p-1)=3$, and thus
$2|S_{g_i}|+ |S_{g_j}| \ge 3 + (2p+1 -(p-1))= p+5$, for each choice of distinct $i,j \in \{1,2,3\}$.

Thus,  if there is a choice of $i,j$ such that $2 g_i +g_j = c$, applying part (2) of Lemma \ref{lemma_main}, yields the claimed result as $2|S_{g_i}|+ |S_{g_j}| \ge p+4$.

By inspection we can see that indeed there always is such a choice. 
To wit, for $d$ an element in $H_1$ such that $H_1 = \langle c \rangle \oplus \langle d \rangle$ we note that 
\[\{ \{g_1, -g_1\},\{g_2, -g_2\},\{g_3, -g_3\} \} =  \{\{ d, -d  \},\{c+d, -c-d \},\{c-d, -c+d \}\}.\] 
There are eight possibilities for the set $\{g_1,g_2, g_3\}$ (note that the order of the elements is not relevant), and for each of these eight choices we find a relation of the form   $2 \cdot g_i + 1 \cdot g_j + 0 \cdot g_k = c$ with $\{i,j,k\}= \{1,2,3\}$. Specifically:
\begin{itemize}
\item $2 \cdot d+1 \cdot (c+d)+ 0 \cdot (c-d)=c$
\item $1 \cdot d +0 \cdot (c+d)+ 2 \cdot (-c+d)=c$
\item $0 \cdot d + 1 \cdot (-c-d)+ 2 \cdot (c-d)=c$
\item $1 \cdot d + 0 \cdot (-c-d)+ 2 \cdot (-c+d)=c$
\item $2 \cdot (-d) + 0 \cdot (c+d)+ 1 \cdot (c-d)=c$
\item $0 \cdot (-d) +2 \cdot (c+d)+ 1 \cdot (-c+d)=c$
\item $1 \cdot (-d) +2 \cdot (-c-d) + 0 \cdot (c-d)=c$
\item $1 \cdot (-d) +2 \cdot (-c-d) + 0 \cdot (-c+d)=c$
\end{itemize}
The claim is established. 
\end{proof}

\subsection{Proof of Theorem \ref{thm_main}}

To establish our main result we combine the partial results obtained thus far. 

By Lemmas \ref{lem_lbsimple} and \ref{lem_lbimproved} we know that $\g(C_3 \oplus C_{3n}) \ge 3n+3$ for each $n \ge 2 $.

Now, assume that $p \ge 5$ is prime. We want to show that  $\g(C_3 \oplus C_{3p}) \le 3p+3$. Let $S$ be a squarefree sequence of length $3p+3$ over  $C_3 \oplus C_{3p}$. 
We need to show that $S$ has a zero-sum subsequence of length $3p$. We continue to use the maps $\pi_1$ and $\pi_2$ introduce in the preceding subsection. 

If $\sigma(\pi_1(S)) \neq 0$, then $S$ has a zero-sum subsequence by Proposition \ref{prop_nonzero}.
If  $\sigma(\pi_1(S)) = 0$, then $S$ has a zero-sum subsequence either by Proposition \ref{prop_zero_full} or by Proposition \ref{prop_zero}. 

Thus, in any case, $S$ has a zero-sum subsequence of length $3p$ and therefore $\g(C_3 \oplus C_{3p}) \le 3p+3$. 
In combination with the lower bound this implies that indeed $\g(C_3 \oplus C_{3p}) =  3p+3$ for each prime  $p \ge 5 $.

It remains to determine the value of $\g(C_3 \oplus C_6)$ and of $\g(C_3  \oplus C_9)$. We know by Lemmas \ref{lem_lbsimple} and \ref{lb_C3C9} that the respective values are lower bounds. To show that these values are the exact values of the Harborth constant we used an algorithm for determining the Harborth constant that we discuss in the last section. 

\section{An algorithm for determining the Harborth constant and some computational results}

For the description of the algorithm we use the language of sets rather than that of sequences, as the description feels slightly more natural. 
To determine the Harborth constant of $G$ means to find the smallest $k$ such that each subset of $G$ of cardinality $k$ has a subset of cardinality $\exp(G)$ with sum $0$.
We outline the algorithm we used below. 

In the first step, all subsets of $G$ of cardinality $\exp(G)$  with sum $0$ are constructed (see Remark \ref{rem_algo} for some details on this). If the subsets of cardinality $\exp(G)$ with sum $0$ happen to be all the subsets of $G$ of cardinality $\exp(G)$, then this means that the Harborth constant is $\exp(G)$. If not, then we consider all subsets of $G$ that are direct successors of a set of cardinality equal to  $\exp(G)$ with sum $0$; in other words, we extend each subset of cardinality $\exp(G)$ with sum $0$ in all possible ways to a subset of cardinality $\exp(G)+1$. Thus, we obtain all subsets of $G$ of cardinality $\exp(G)+1$ that contain a subset of cardinality $\exp(G)$ with sum $0$. If the subsets of $G$ obtained in this way are all subsets of $G$ of cardinality $\exp(G)+1$, then we have established that the Harborth constant of $G$ is $\exp(G)+1$. If not, then we continue as above until for some $k$ the set of subsets of cardinality $k$ obtained in this way coincides with the set of all subsets of cardinality $k$ of $G$. 
   
Below we detail the steps of the algorithm a bit more. However, a more complete investigation of the algorithmic problem will be presented elsewhere, and we gloss over more technical aspects here.    

\subsection{The steps of the algorithm}  
\ \\ [1ex]

\noindent
\textbf{Input:} A finite abelian group $G$ of order $n$ and exponent $e$. \\  
\textbf{Output:}  $\g(G)$, the Harborth constant of the group $G$.

\begin{itemize} 
 \item \ [Initialization] Let $\Z(e)$ denote the collection of all subsets of $G$ of cardinality $e$ that have sum $0$. Set $k = e$. 
 \item \ [Check] If $|\Z(k)| =  \binom{n}{k}$, then return $\g(G)=k$ and end. Else, increment $k$ to $k+1$.   
 \item \ [Extend] Let $\Z(k)$ denote the collection of all subsets of cardinality $k$ of $G$ that have subset  that is in $\Z(k-1)$. Go to [Check].
\end{itemize}

We add some further explanations and remarks. 

\begin{remarks}
\label{rem_algo}
  \ 
\begin{enumerate}
\item The group intervenes only in the step [Initialization]. (The rest of the algorithm operates merely with subsets of a given ambient set.) To find all subsets of cardinality $e$ with sum $0$, we browse all subsets of cardinality $ e-1 $. For each of these sets, we check if the inverse of the $e-1$ elements belongs to the set; if it does not, we add it to the set to obtain a set of cardinality $e$ and sum $0$.
For this step the subsets of $G$ are represented by a bitmap.
\item  For the latter parts of the algorithm in addition to the representations as bitmaps a judiciously chosen numbering of the subsets of $G$ is used. 
In particular, the numbering is chosen in such a way that for every subset its cardinal is at least as large as the cardinal of all its predecessors.
This is useful as in this way at each step, our search can be efficiently limited to the $ \binom{n}{k}$ subsets of a given cardinal $k$ instead of having to consider all $2^n$ subsets at each step. 
\item The subsets of cardinality $k$ that are not in $\Z(k)$  are all the subsets of $G$ of cardinality $k$ that have no zero-sum subset of $e$ elements. 
Thus, in the final step before the algorithm terminates we effectively have all the subsets 
of $G$ of cardinality $\g(G)-1$ that have no zero-sum subset of $e$ elements. That is, the algorithm can be immediately modified to solve the inverse problem associated to $\g(G)$ as well.  
\item The algorithm is valid for any finite abelian group. With the hardware at out disposal it is possible to compute the Harborth constant for finite abelian groups of order up to about $45$. The main limiting factor is memory.  In order to increase the size of accessible groups, we are currently working on a more efficient subset-representation based on data compression.  
\item The fact that $e$ is equal to the exponent of the group is not relevant for the algorithm. It can be directly modified to compute related constants. 
 \end{enumerate}
 \end{remarks}

We end by mentioning two further computational results. 

\begin{proposition} 
\label{prop_comp}
\
\begin{enumerate}
\item $\g(C_6 \oplus C_6) = 13$. 
\item $\g(C_3 \oplus C_{12}) = 15$. 
\end{enumerate}
\end{proposition}

The former confirms the conjecture by Gao and Thangadurai, mentioned in the introduction, $\g(C_n \oplus C_n) = 2n+1$ for even $n$ in case $n=6$.
The latter shows that  $\g(C_3 \oplus C_{3n}) = 3n+3$ also holds for $n=4$, which supports the idea that $\g(C_3 \oplus C_{3n}) = 3n+3$ might hold for $n$ that are not prime as well.

\end{document}